\documentclass[11pt,twoside,a4paper]{amsart} 
\usepackage[top=2.54cm,bottom=3cm,left=2.0cm,right=2.54cm]{geometry} 
\usepackage[english]{babel}                      
\usepackage[utf8]{inputenc}
\usepackage[T1]{fontenc}     
\usepackage{amsmath,amssymb,amsthm}
\usepackage{setspace} 
\usepackage{indentfirst}
\usepackage{paralist}
\usepackage{hyperref}

\renewenvironment{proof}{{\noindent \sc Proof:}}{\begin{flushright}$\blacksquare$\end{flushright}}

\newtheorem{teo}{Theorem}

\newtheorem*{q}{Question}
\newtheorem{col}[teo]{Corollary}

\newtheorem{lema}[teo]{Lemma}
\newtheorem*{ch}{Hypercyclicity Criterion}
\newtheorem*{cc}{Cyclicity Criterion}
\newtheorem*{mteo}{Main Theorem}
\theoremstyle{definition}

\newtheorem*{notacao}{Notation}

\DeclareMathOperator{\spn}{span}
\DeclareMathOperator{\orb}{orb}

\DeclareMathOperator{\im}{ran}

\newcommand{\norm}[1]{\left\lVert#1\right\rVert}
\newcommand{\abs}[1]{\left\lvert#1\right\rvert}
\newcommand{\eqdef}{\mathrel{\mathop:}=}
\newcommand{\restr}[2]{{\left.\kern-\nulldelimiterspace  #1\vphantom{\big|}\right|_{#2}}}

\title{Relationships between Cyclic and Hypercyclic Operators}
\author{André Augusto and Leonardo Pellegrini}
\thanks{The research of the first author was supported by CNPQ, grant 142035/2018-1.}
\address{Departamento de Matemática, Universidade de São Paulo}
\email{andreqa@ime.usp.br; leonardo@ime.usp.br}

\begin{document}
\begin{abstract}
A bounded linear operator $T$ on a Banach space $X$ is called hypercyclic if there exists a vector $x \in X$ such that $\orb{(x,T)}$ is dense in $X$. The Hypercyclicity Criterion is a well-known sufficient condition for an operator to be hypercyclic. One open problem is whether there exists a space where the Hypercyclicity Criterion is also a necessary condition. For a number of reasons, the spaces with very-few operators are some natural candidates to be a positive answer to that problem. In this paper, we provide a theorem that establishes some relationships for operators in these spaces.
\end{abstract}
\sloppy
\maketitle

\thispagestyle{empty}

\section*{Introduction}

A bounded linear operator $T$ on a separable Banach space $X$ is {\it hypercyclic} if there exists a vector $x \in X$ such that $\orb{(x,T)} \eqdef \{T^nx \, : \, n \geq 0\}$ is dense in $X$. Such vector $x$ is called a {\it hypercyclic vector} for $T$.

One of the most important results about hypercyclic operators is the well-known {\it Hypercyclicity Criterion}, stated below:

\begin{ch} Let $T$ be any operator on a separable Banach space $X$. Suppose that there exists two dense subsets $X_0, Y_0 \subseteq X$, an increasing sequence $(n_k)_{k \geq 1}$ of positive integers, and a sequence $(S_{n_k})_{k \geq 1}$ of maps (not necessarily linear nor continuous) $S_{n_k} \; : \, Y_0 \to X$ such that
\begin{enumerate}[(i)]
\item for every $x \in X_0$, $T^{n_k}x \to 0$;
\item for every $y \in Y_0$, $S_{n_k}y \to 0$;
\item for every $y \in Y_0$, $T^{n_k}S_{n_k}x \to x$.
\end{enumerate}

Then $T$ is hypercyclic.
\end{ch}

For a long time it was an open problem whether the Hypercyclicity Criterion was also a necessary condition for an operator to be hypercyclic or not. In \cite{read}, De la Rosa and Read showed that this was not the case, constructing a suitable space with hypercyclic operators that does not satisfy the Hypercyclicity Criterion.

In \cite{bayart}, Bayart and Matheron showed that such operators also exist on $\ell_p(\mathbb{N})$ spaces ($1 \leq p < \infty$) and $c_0(\mathbb{N})$. Since these spaces are, in some way, the building blocks in Banach space theory, it was natural to ask the following question:

\begin{q} Does exist a Banach space such that every hypercyclic operator necessarily satisfies the Hypercyclicity Criterion?
\end{q}

Every known example of a hypercyclic operator of the form $\lambda I + K$ satisfy the Hypercyclicity Criterion. With that in mind, one category of spaces that are a natural candidate to be a positive answer to the question above are the {\it spaces with very-few operators}. In these spaces, every bounded linear operator can be written as $\lambda I + K$, with $\lambda \in \mathbb{K}$ and $K$ compact - this property is known as {\it the scalar-plus-compact property}. However, since it is not known if every operator with this form also satisfy the criterion, there is not an answer to the question yet. 

One way to approach the previous question is to notice that if $T = \lambda I + K$ is hypercyclic, then $K$ is a cyclic operator.\footnote{Remember that a bounded linear operator $T$ on a separable Banach space $X$ is {\it cyclic} if there exists a vector $x \in X$ such that $\spn\orb{(x,T)} \eqdef \spn \{T^nx \, : \, n \geq 0\}$ is dense in $X$. Such vector $x$ is called a {\it cyclic vector} for $T$.} So it is now natural to try and establish some results concerning cyclic operators.

In \cite{grivaux}, Grivaux stated a {\it Cyclicity Criterion}, being inspired by its hypercyclic counterpart:

\begin{cc} Let $T$ be any operator on a separable Banach space $X$. Suppose that there exist two dense subsets $V$ and $W$ of $X$, a sequence $(p_k)_{k \geq 1}$ of polynomials, and a sequence $(S_k)_{k \geq 1}$ of maps (not necessarily linear nor continuous) $S_k \; : \, W \to X$ such that
\begin{enumerate}[(i)]
\item for every $x \in V$, $p_k(T)x \to 0$;
\item for every $x \in W$, $S_kx \to 0$;
\item for every $x \in W$, $p_k(T)S_kx \to x$.
\end{enumerate}

Then $T \oplus T$ is cyclic.
\end{cc}

It is clear that if an operator satisfy the Hypercyclicity Criterion, then it also satisfy the Cyclicity Criterion.

With the goal of establish some relationships between a hypercyclic operator $T = \lambda I + K$ and the cyclic operator $K$ itself, we obtained the following theorem, which is the main result of this paper:

\begin{mteo} \label{newteo} Let $T$ be a hypercyclic operator on a Banach space $X$ such that $T = \lambda I + S$, with $\lambda \in \mathbb{K}$. The following assertions are equivalent:
\begin{enumerate}[(i)]
\item $T$ satisfies the Hypercyclicity Criterion;
\item $T$ satisfies the Cyclicity Criterion;
\item $S$ satisfies the Cyclicity Criterion;
\item $S \oplus S$  is cyclic;
\item $T \oplus T$  is cyclic;
\item $T \oplus T$ is hypercyclic.
\end{enumerate} 
\end{mteo}

Some of the equivalences in the theorem above are well-known and we will be discussing this ahead. Also, the operator $S$ in the theorem above does not need to be compact.

Using our Main Theorem above, one have several ways to prove that a hypercyclic operator on a space with very-few operators satisfy the Hypercyclicity Criterion.

\begin{notacao}
Throughout this paper $X$ will denote an infinite dimensional separable Banach space, $\mathcal{B}(X)$ the algebra of bounded linear operators on $X$,  $\mathbb{K}[x]$ is the polynomial ring over the field $\mathbb{K}$ and $\deg{p}$ the degree of the polynomial $p \in \mathbb{K}[x]$.
\end{notacao}

\section{Main Theorem's Proof}

As we said in the introduction, some of the equivalences on our Main Theorem are well-known. In fact, this next theorem, which is a compilation of these previous known results, was the inspiration behind our Main Theorem:

\begin{teo} \label{teoequiv} Let $T \in \mathcal{B}(X)$ be hypercyclic. The following assertions are equivalent:
\begin{enumerate}[(i)]
\item $T$ satisfies the Hypercyclicity Criterion;
\item $T \oplus T$ is hypercyclic;
\item $T \oplus T$ is cyclic;
\item $T$ satisfies the Cyclicity Criterion.
\end{enumerate}
\end{teo}
\begin{proof} The equivalence between $(i)$ and $(ii)$ is due to Bès and Peris \cite[Theorem 2.3]{equiv}. Since the adjoint $T^*$ of an hypercyclic operator has no eigenvalues (this result is due to Bès \cite[Lemma 1]{bes}), then the equivalence between $(iii)$ and $(iv)$ follows from a result of Grivaux \cite[Proposition 5.3]{grivaux}. Finally, the equivalence between $(ii)$ and $(iii)$ is also due to Grivaux \cite[Propositon 4.1]{grivaux}.
\end{proof}

Our Main Theorem would easily follow from the theorem above if we had that $S = T - \lambda I$ is hypercyclic whenever $T$ is. However, this is not the case: as we mentioned in the introduction, there are hypercyclic operators $T$ such that $T = \lambda I + K$, with $K$ being a compact operator - and it is known that compact operators are not hypercyclic (see \cite[Theorem 5.11]{LC}).

However, one can notice that Theorem \ref{teoequiv} will still help us prove our Main Theorem; being more specific, the equivalences between $(i)$, $(v)$ and $(vi)$ on the Main Theorem are all already proved by Theorem \ref{teoequiv}. 

That being said, the remaining equivalences will follow from the next lemma:

\begin{lema} \label{pol} Let $T, S \in \mathcal{B}(X)$ such that $T = \lambda I + S$. Then $\mathbb{K}[T] = \mathbb{K}[S]$.
\end{lema}
\begin{proof} It is clear that $\mathbb{K}[T] \subseteq \mathbb{K}[S]$. We will show the other inclusion by induction on the polynomial degree.

Let $p(S) \in \mathbb{K}[S]$.  If the degree of $p$ is 1, then $p(S) = a_0I + a_1S = (a_0 - a_1\lambda)I + a_1(\lambda I + S) = (a_0 - a_1\lambda)I + a_1T = q(T) \in \mathbb{K}[T]$.

Assume now that every polinomial with degree less or equal to $n$ belongs to $\mathbb{K}[T]$. Consider $p(S) = a_0I + a_1S + \ldots + a_nS^n + a_{n+1}S^{n+1}$. We have that $$a_{n+1}T^{n+1} = a_{n+1}(\lambda I + S)^{n+1} = q(S) + a_{n+1}S^{n+1}$$ with $\deg{q} \leq n$. Then, $a_{n+1}S^{n+1} = a_{n+1}T^{n+1} - q(S)$ which implies that \begin{align*}
p(S) & = a_0I + a_1S + \ldots + a_nS^n + a_{n+1}S^{n+1} \\
& = a_0I + a_1S + \ldots + a_nS^n + (a_{n+1}T^{n+1} - q(S)) \\
& = [a_0I + a_1S + \ldots + a_nS^n - q(S)] + a_{n+1}T^{n+1}
\end{align*}

Defining $r(S) \eqdef a_0I + a_1S + \ldots + a_nS^n - q(S)$, we have that $\deg r(S) \leq n$. By our induction hypothesis, $r(S) \in \mathbb{K}[T]$.

Since $\mathbb{K}[T]$ is a ring and $r(S), a_{n+1}T^{n+1} \in \mathbb{K}[T]$, then $r(S) + a_{n+1}T^{n+1} = p(S) \in \mathbb{K}[T]$, as desired.
\end{proof}

With that lemma in hand, the proof of our Main Theorem is quite straightfoward: 

\subsection*{Proof of the Main Theorem}

\noindent \underline{(i) $\Rightarrow$ (ii)}. Trivial. \\

\noindent \underline{(ii) $\Rightarrow$ (iii)}. By hypothesis, there exists $V, W \subseteq X$ dense subsets, $(p_k)_{k \geq 1}$ a sequence of polynomials and $(S_k)_{k \geq 1}$ a sequence of maps that satisfy conditions $(i), (ii)$ and $(iii)$ of the Cyclicity Criterion for $T$. By Lemma \ref{pol}, for each $k \geq 1$ there is some $q_k \in \mathbb{K}[x]$ such that $p_k(T) = q_k(S)$. Hence, using this equality, the same subsets $V, W \subseteq X$ and maps $(S_k)_{k \geq 1}$ that was used for $T$, it easily follows that $S$ satisfy the Cyclicity Criterion as well. \\

\noindent \underline{(iii) $\Rightarrow$ (iv)}. Straightfoward from the Cyclicity Criterion. \\

\noindent \underline{(iv) $\Rightarrow$ (v)}. It also follows from Lemma \ref{pol}. Indeed, suppose that $S \oplus S$ is cyclic with cyclic vector $(x,y)$. Thus, given nonempty open sets $(U,V)$ there exists a polynomial $p$ such that $p(S)x \in U$ and $p(S)y \in V$. By Lemma \ref{pol}, there exists a polynomial $q$ such that $q(T) = p(S)$. Thus, we have that $q(T)x \in U$ and $q(T)y \in V$, which shows that $T \oplus T$ is cyclic with cyclic vector $(x,y)$. \\
 
\noindent \underline{(v) $\Rightarrow$ (vi) and (vi) $\Rightarrow$ (i)}. \smallskip Theorem \ref{teoequiv}.

\begin{flushright}$\blacksquare$\end{flushright}

One can notice that the equivalence between $(ii)$ and $(iii)$ relies only on Lemma \ref{pol}, not using at all the hypothesis that $T$ is hypercyclic.

A theorem by Bourdon \cite{bourdon} states that if $p \in \mathbb{K}[x]$ is a nonzero polynomial and $T$ is a hypercyclic operator, then $p(T)$ has dense range. This result, along with Lemma \ref{pol}, yields the next corollary:

\begin{col} If $T = \lambda I + S$ is hypercyclic, then $S$ is cyclic with dense range. Also, if $p \in \mathbb{K}[x]$ is a nonzero polynomial, then $p(S)$ has dense range.
\end{col}
\begin{proof} Notice that $S = T - \lambda I = p(T)$, where $p(x) = x - \lambda$. Thus, $p(T) = S$ has dense range.

The second part is similar: if $p \in  \mathbb{K}[x]$ is a nonzero polynomial, then by Lemma \ref{pol} there exists a nonzero polynomial $q \in \mathbb{K}[x]$ such that $p(S) = q(T)$. Hence, since $q(T)$ has dense range, we have that $p(S)$ also has dense range.
\end{proof}

Every hypercyclic operator has dense range (as an easy consequence of Bourdon's Theorem). However, this is not the case for cyclic operators. Indeed, let $F$ be the forward shift on $\ell_2$. It is not hard to see that $F$ is a cyclic operator that does not have dense range (since $e_1 \not\in \im{F}$). Thus, every operator $\lambda I + F$ is not hypercyclic.

As far the second statement in the corollary goes, the same argument can be used: if $T$ is a cyclic operator, it is not necessarily true that $p(T)$ has dense range, for any $p \in \mathbb{K}[x]$.

With that in mind, the previous corollary establishes that when an operator $S$ is cyclic and such that $\lambda I + S$ is hypercyclic, then $S$ itself share some properties that usually only hypercyclic operators have. Now, it is fair to ask if other properties from hypercyclic operators are also shared with $S$. For example, is it known that if $T$ is hypercyclic, then $T^n$ is as well (this is known as Ansari's Theorem, see \cite[Theorem 1]{ansari}). In the case illustrated above, would $S^n$ be cyclic?

\section{The Bayart/Matheron Construction on Banach Spaces with Very-few Operators}

As we said in the introduction, Bayart and Matheron constructed in \cite{bayart} operators on $\ell_p$ and $c_0$ that does not satisfy the Hypercyclicity Criterion. In order to answer the question we raised before, it is natural to try to replicate this construction on Banach spaces with very-few operators. To facilitate what we are going to discuss next, we state below the theorem proved by Bayart and Matheron. We remember that the {\it forward shift associated to} $(e_i)$, where $(e_i)_{i \geq 0} \subseteq X$ is a linearly independent sequence, is the linear operator $S \, : \, E \to E$ defined by $S(e_i) = e_{i+1}$, where $E = \spn\{e_i \, : \, i \geq 0\}$.

\begin{teo}[Theorem 1.1 on \cite{bayart}] \label{teobayart} Let $X$ be a Banach space. Assume $X$ has a normalized unconditional basis $(e_i)_{i \geq 0}$ whose associated forward shift is continuous. Then there exists a hypercyclic operator $T \in \mathcal{B}(X)$ such that $T \oplus T$ is not hypercyclic.
\end{teo}

One immediate problem arises when we try to use the theorem above on Banach spaces with very-few operators: it is well known that these spaces do not admit an unconditional basis (in particular, they do not admit unconditional basic sequences). Since the unconditional basis was used in \cite{bayart} to prove that the operator $T$ constructed in the theorem above was continuous, it is reasonable to try to find other ways to prove that continuity. Not only that, but since there are so few known Banach spaces with very-few operators\footnote{The only spaces with very-few operators known by the authors are the famous Argyros-Haydon space \cite{AH}, the $X_{\mathfrak{nr}}$ space constructed by Argyros and Motakis \cite{AM} and the $X_{Kus}$ space constructed by Manoussakis, Pelczar-Barwacz and Swietek \cite{MPS}.} and they all follow the Bourgain-Delbaen method in their construction, we could use some properties (other than the scalar-plus-compact one) of these specific spaces and their constructions to prove the continuity of $T$.

%

We will show, however, that the additional properties of these known spaces are not enough: if a space with very-few operators $X$ has a Schauder basis (not necessarily unconditional), the operator $T$ constructed in Theorem \ref{teobayart} can not be continuous.

In order to show that, let us recall how the operator $T$ was constructed in \cite{bayart}. Let $(b_n)_{n  \geq 1}$ be an increasing sequence of positive numbers. If $(e_n)_{n \geq 0}$ is a normalized basis for the Banach space $X$, we define $T$ such as \begin{equation} \label{eqdefT}
T(e_i) = \begin{cases}
2e_{i+1} & \textnormal{  if  } i \in [b_{n-1}, b_n - 1[, \textnormal{  for some  } n \geq 1 \\
\varepsilon_ne_{b_n} + f_n &  \textnormal{  if  } i = b_n - 1, \textnormal{  for some  } n \geq 1 \\
\end{cases}
\end{equation} where $\varepsilon_n \in \mathbb{K}$ and $f_n \in X$.

The precise defitinion of $\varepsilon_n$ and $f_n$ will not be needed right now. Suppose that $T$ is continuous and $X$ has very-few operators. Hence, we may write $T = \lambda I + K$, for some $\lambda \in \mathbb{K}$ and $K \in \mathcal{B}(X)$ compact. We will show that the bounded sequence $(e_{j_k})_{k \geq 1} \subseteq  (e_n)_{n \geq 0}$, where $j_k \neq b_n - 1$ for every $n \geq 1$, is such that $(Ke_{j_k})_{k \geq 1}$ does not have a convergent subsequence, an obvious contradiction.

We have that $Ke_{j_k} = (T - \lambda I)(e_{j_k}) = 2e_{{j_k}+1} - \lambda e_{j_k}$, since $j_k \neq b_n - 1$. Let us recall now that if $(x_n)_{n \in \mathbb{N}}$ is a basis for a Banach space $(X, \norm{\cdot})$ and $x = \sum_{n=1}^{\infty} \alpha_nx_n$, then $\norm{x}_{(x_n)} = \sup_{k}\norm{\sum_{n=1}^{k} \alpha_nx_n}$ is an equivalent norm to $\norm{\cdot}$.

Let now $(Ke_{j_{k_i}})_{i \geq 1} \subseteq (Ke_{j_k})_{k \geq 1}$ be any subsequence. If $\lambda \neq 0$, we have: \begin{align*}
\norm{Ke_{j_{k_i}} - Ke_{j_{k_l}}}  & \geq M \norm{Ke_{j_{k_i}} - Ke_{j_{k_l}}}_{(e_n)}  = M\norm{2e_{j_{k_i}+1} - \lambda e_{j_{k_i}} - 2e_{{j_{k_l}+1}} + \lambda e_{j_{k_l}}}_{(e_n)} \\
& \geq M\norm{- \lambda e_{j_{k_i}} + 2e_{j_{k_i}+1} + \lambda e_{j_{k_l}} - 2e_{{j_{k_l}+1}} }_{(e_n)} \\
& \geq M \abs{\lambda}
\end{align*} for $i > l$. If $\lambda = 0$, we have:
\begin{equation} \label{eqbayart}
\norm{Ke_{j_{k_i}} - Ke_{j_{k_l}}} \geq  M\norm{2e_{j_{k_i}+1} - 2e_{{j_{k_l}+1}}}_{(e_n)} \geq 2M
\end{equation} In any case, we have that $\norm{Ke_{j_{k_i}} - Ke_{j_{k_l}}} \not\to 0$ and, therefore, such subsequence does not converge.

One can ask if we can modify the definition of $T$ in order to fix the problem above. In subsequent works, Bayart and Matheron simplified the original construction done in \cite{bayart}. In Section 4.2 of their book \cite{DLO}, they showed how they improved the original construction: instead of defining $T(e_i) = 2e_{i+1}$ for $i \in [b_{n-1}, b_n - 1[$, they define $T(e_i) = w(i+1)e_{i+1}$, where $(w(i))_{i \geq 1}$ is an increasing sequence of positive numbers such that $$w(i) \eqdef 4\left(1 - \frac{1}{2\sqrt{i}}\right)$$ for every $i \geq 1$.

With this definition of $T$ in mind, the proof of the non-continuity of $T$ should change. However, it is clear that the only change would occur in Equation \eqref{eqbayart}. We would have \begin{equation} \label{eqbayart2}
\norm{Ke_{j_{k_i}} - Ke_{j_{k_l}}} \geq w(j_{k_l})M
\end{equation} and since $(w(i))_{i \geq 1}$ is an increasing sequence such that $w(i) \to 4$, then $T$ is still non-continuous.

However, one can again ask if we could not change the sequence $(w(i))_{i \geq 1}$ for one that is decreasing and such that $w(i) \to 0$. This would eliminate the contradiction obtained in Equation \eqref{eqbayart2} and would leave open the question of whether $T$ is continuous or not.

Looking again at the Bayart and Matheron construction, even if we choose $(w(i))_{i \geq 1}$ decreasing and such that $w(i) \to 0$, the operator $T$ would still not be continuous. Although this would solve the problems encountered in Equation \eqref{eqbayart2}, this choice would make either $(\varepsilon_n)_{n \geq 1}$ or $(f_n)_{n \geq 1}$ unbounded. Hence, $T$ would not be continuous: this time, $T(e_{b_n - 1}) = \varepsilon_ne_{b_n} + f_n$ would be the expression that makes $T$ not continuous (see the definition of $T$ in \eqref{eqdefT}).

Since this last expression is the one that makes $T$ hypercyclic - and Bayart and Matheron did lots of work to make that happen - it is now better to have other ideas than to fix this approach.

\section*{Final Remark}

This article is part of the first author's PhD thesis, written under the supervision of the second author.


\begin{thebibliography}{99}

\bibitem{ansari} {\sc Ansari, S.I.,} {\it ``Hypercyclic and Cyclic Vectors''}, J. Funct. Anal. {\bf 128} (1995), 374–383.

\bibitem{AH} {\sc Argyros, S.} e {\sc Haydon, R.,} {\it ``A Hereditarily Indecomposable} $\mathcal{L}_\infty${\it -space that Solves the Scalar-Plus-Compact Problem''}, Acta Math. {\bf 206} (2011), 1–54.

\bibitem{AM} {\sc Argyros, S.} e {\sc Motakis, P.,} {\it ``The Scalar-Plus-Compact Property in Spaces without Reflexive Subspaces''}, Trans. Amer. Math. Soc. {\bf 371} (2019), no. 3, 1887–1924.

\bibitem{DLO} {\sc Bayart, F.} e {\sc Matheron É.,} {\it Dynamics of Linear Operators}. Series {\it Cambridge Tracts in Mathematics}, {\bf 179}. Cambridge: New York, 2009.

\bibitem{bayart} {\sc Bayart, F.} e {\sc Matheron É.}, {\it ``Hypercyclic Operators Failing the Hypercyclic Criterion on Classical Banach Spaces''}, J. Funct. Anal. {\bf 250} (2007), 426–441.

\bibitem{bes} {\sc Bès, J.,} {\it ``Invariant Manifolds of Hypercyclic Vectors for the Real Scalar Case''}, Proc. Amer. Math. Soc. {\bf 127} (1999), 1801–1804. 

\bibitem{equiv} {\sc Bès, J.} e {\sc Peris, A.,} {\it ``Hereditarily Hypercyclic Operators''}, J. Funct. Anal. {\bf 167} (1999),  94–112.

\bibitem{bourdon} {\sc Bourdon, P.,} {\it ``Invariant Manifolds of Hypercyclic Operators}, Proc. Am. Math. Soc. {\bf 118} (1993), 845–847.

\bibitem{read} {\sc De La Rosa, M.,} e {\sc Read, C.J.,} {\it ``A Hypercyclic Operator whose Direct Sum T} $\oplus$ {\it T is not Hypercyclic''}. J. Operator Theory {\bf 61} (2009), no. 2, 369–380.

\bibitem{grivaux} {\sc Grivaux, S.,} {\it ``Hypercyclic Operators, Mixing Operators, and the Bounded Steps Problem''}, J. Funct. Anal. {\bf 202} (2003), 486–503.

\bibitem{LC} {\sc Grosse-Erdmann, K.} e {\sc Peris, A.,} {\it Linear Chaos}. Coleção {\it Universitext}. Springer: Londres, 2011.

\bibitem{MPS} {\sc Manoussakis, A.,}  {\sc Pelczar-Barwacz, A.} e {\sc Swietek, M.,} {\it ``An Unconditionally Saturated Banach Space with the Scalar-Plus-Compact Property''}, J. Funct. Anal. {\bf 272} (2017), 4944–4983.



\end{thebibliography}
\end{document}